\documentclass[12pt,conference]{IEEEtran}
\ifCLASSOPTIONcompsoc
  \usepackage[nocompress]{cite}
\else
  \usepackage{cite}
\fi

\usepackage{etex}
\usepackage{amsmath,amssymb,url}
\usepackage{fdsymbol}
\usepackage{bussproofs}
\usepackage{epic}
\usepackage{eufrak}
\usepackage{eepic}
\newsavebox{\wwide}
\newcommand{\wwidehat}[1]{\sbox{\wwide}{$#1$}
\ifdim\wd\wwide < 1.1 em \widehat{#1} \else
\setlength
{\unitlength}{0.01\wd\wwide}\overset
{\begin{picture}(100,6)
\path(0,0)(50,6)(100,0)
\end{picture}}{#1}\fi}

\newcommand{\wwidetilde}[1]{\sbox{\wwide}{$#1$}
\ifdim\wd\wwide < 1.1 em \widetilde{#1} \else
\setlength
{\unitlength}{0.01\wd\wwide}\overset
{\begin{picture}(100,6)
\path(0,0)(33,6)(45,6)(55,0)(67,0)(100,6)
\end{picture}}{#1}\fi}

\ifx\pdfoutput\@undefined\usepackage[usenames,dvips]{color}
\else\usepackage[usenames,dvipsnames]{color}
\IfFileExists{pdfcolmk.sty}{\usepackage{pdfcolmk}}{} 
\fi

\definecolor{ChadDarkBlue}{rgb}{.1,0,.2}  
\definecolor{ChadBlue}{rgb}{.1,.1,.5}  
\definecolor{ChadRoyal}{rgb}{.2,.2,.8}  
\definecolor{ChadGreen}{rgb}{0,.4,0}    
\definecolor{ChadRed}{rgb}{.5,0,.5}  
\usepackage{amsthm}
\usepackage{epic,eepic}
\usepackage[mathscr]{euscript}
\usepackage{enumerate}

\usepackage{diagrams}
\usepackage{graphicx}

\usepackage{lineno}
\usepackage{comment}

\usepackage{tikz}
\usetikzlibrary{graphs}
\usetikzlibrary{positioning}

\def\smallskip{\vskip\smallskipamount}
\def\medskip{\vskip\medskipamount}
\def\bigskip{\vskip\bigskipamount}

\numberwithin{equation}{section}

\theoremstyle{plain}
\newtheorem{theorem}{Theorem}[section]
\newtheorem{Theorem}{Theorem}[section]

\newtheorem{proposition}[theorem]{Proposition}
\newtheorem{Proposition}[theorem]{Proposition}

\newtheorem{Corollary}[theorem]{Corollary}

\theoremstyle{definition}

\newtheorem{Definition}[theorem]{Definition}

\newtheorem{remark}[theorem]{Remark}

\DeclareTextCommand{\elqq}{T1}{\leavevmode\char16\nobreak\hskip0pt}
\DeclareTextCommand{\erqq}{T1}{{\edef\@SF{\spacefactor\the\spacefactor}%
\nobreak\char17\@SF\relax}}

\newcommand{\uv}[1]{``{#1}"}

\newcommand{\lcr}{\Mapsto}
\newcommand{\eqcr}{\Mapsfrom\hskip-0.1521em\Mapsto}

\newcommand\ekviv{{\ \Longleftrightarrow\ }}

\newcounter{ok}
{\end{list}}

\newcounter{aok}
{\end{list}}

\def\go#1;#2;#3 {\vbox to0pt{\kern-#3\hbox{\kern#2 #1}\vss}\nointerlineskip}

  \newcommand{\myneco}[1]{\textcolor[rgb]{1.00,0.20,0.70}{#1}}

\begin{document}
\title{The groupoid-based logic for lattice effect algebras}
\author{\IEEEauthorblockN{Ivan~Chajda}
\IEEEauthorblockA{Department of Algebra and Geometry  \\
                              Faculty of Science,  Palack\'y University Olomouc\\
                             17. listopadu 12,  CZ-771 46 Olomouc, Czech Republic\\
        E-mail: ivan.chajda@upol.cz}\and
\and 
\IEEEauthorblockN{Helmut L\"anger}
\IEEEauthorblockA{Institute of Discrete Mathematics and Geometry\\
		Faculty of Mathematics and Geoinformation,  TU Wien\\
		{Wiedner Hauptstra\ss e 8-10}, 1040 Vienna, Austria\\
        E-mail: helmut.laenger@tuwien.ac.at}
\and 
\IEEEauthorblockN{Jan~Paseka}
\IEEEauthorblockA{Department of Mathematics and Statistics\\
		Faculty of Science, Masaryk University\\
		{Kotl\'a\v r{}sk\' a\ 2}, CZ-611~37~Brno, Czech Republic\\
        E-mail: paseka@math.muni.cz}}



\markboth{Groupoid-based logic for lattice effect algebras}%
{Groupoid-based logic for lattice effect algebras}
%

%




\IEEEcompsoctitleabstractindextext{%
\begin{abstract}
The aim of the paper is to establish a certain logic corresponding to lattice effect algebras. 
First, we answer a natural question whether  a lattice effect algebra can be represented 
by means of a groupoid-like structure. We establish a one-to-one correspondence 
between lattice effect algebras and certain groupoids with an antitone involution. 
Using these groupoids, we are able to introduce a suitable logic for lattice effect algebras.

\end{abstract}

}

\maketitle  

\IEEEdisplaynotcompsoctitleabstractindextext

%

{\section*{Introduction}}

\label{intro}
\IEEEPARstart{T}{wo} equivalent quantum structures, D-posets and effect algebras were
introduced in the nineties of the twentieth century. These
were considered as "unsharp" generalizations of the
structures which arise in quantum mechanics, in particular, of orthomodular
lattices and MV-algebras. Effect algebras aim to
describe "unsharp" event structures in quantum mechanics
in the language of algebra.

Effect algebras are fundamental in investigations of fuzzy probability
theory too. In the fuzzy probability frame, the elements of an effect algebra
represent fuzzy events which are used to construct fuzzy random variables.

Effect algebras were introduced by Foulis and Bennett (\cite{FoBe}) and D-posets 
by Chovanec and  K\^o{}pka (\cite{dpos}). 
Although the definition of an effect algebra looks elementary, these algebras have 
several very surprising properties. Concerning these properties the reader is referred to 
the monograph \cite{dvurec}  by Dvure\v censkij and Pulmannov\'a. In particular, 
every effect algebra induces a natural partial order relation and thus can be considered 
as a bounded poset. If this poset is a lattice, the effect algebra is called a 
{\em lattice effect algebra}. A representation of lattice effect algebras by means 
of so-called basic algebras was derived in \cite{CHK}.

Since effect algebras describe quantum effects and are determined by behaviour 
of bounded self-adjoint operators on the Hilbert space of the corresponding 
physical system, we hope that a logic which should be reached by means 
of these algebras will enable us a better understanding of the logic of quantum mechanics.

The aim of the paper is to establish a certain logic corresponding to lattice effect algebras. 
By a logic we mean here a set of formulas in the language of lattice effect algebras enriched 
by logical connectives with a finite set of derivation rules. It can be noticed that for basic algebras 
the same task was solved in \cite{Ch} and for the so-called 
dynamic De Morgan algebras in \cite{TODDMA2,TODDMA3}. 
Since effect algebras are only partial algebras, it looks as an advantage to use another 
algebraic structure which has everywhere defined operations and which is in a one-to-one 
correspondence with the given effect algebra ${\mathbf E}$. 
If ${\mathbf E}$ is lattice ordered, this is possible 
and the corresponding structure can be e.g. the so-called effect near 
semiring, see e.g. \cite{CL}  
for details. However, we can derive another algebra which has only one binary operation, 
i.e., a groupoid enriched by unary and nullary operations. Such approach enables us 
to reduce the set of formulas and the set of derivation rules. This is our aim in the first part 
of the paper. Using this groupoid which is called {\em  effect groupoid}, we are able to introduce 
a suitable logic for lattice effect algebras which is provided in the second part.


\medskip

\section{{Preliminaries} and basic facts}
\label{prelim}

We refer the reader to \cite{Balbes} 
for standard definitions and notations for lattice structures. 

We start with the definition of an effect algebra.

\begin{Definition}\label{defea}
An {\em effect algebra} is a partial algebra $\mathbf E=(E;\oplus,0,1)$ 
of type $(2,0,0)$ satisfying conditions {\rm(E1)} -- {\rm(E4)} for all $x,y,z\in E$:
\begin{enumerate}[\hspace{6pt}(E1)]
\item[{\rm(E1)}] If $x\oplus y$ exists, so does $y\oplus x$ and $x\oplus y=y\oplus x$;
\item[{\rm(E2)}] if $x\oplus y$ and $(x\oplus y)\oplus z$ exist, so do $y\oplus z$ and $x\oplus(y\oplus z)$ and $(x\oplus y)\oplus z=x\oplus(y\oplus z)$;
\item[{\rm(E3)}] there exists a unique $x'\in E$ such that $x\oplus x'$ is defined and $x\oplus x'=1$;
\item[{\rm(E4)}] If $x\oplus1$ exists then $x=0$.
\end{enumerate}
\end{Definition}

Since $'$ is a unary operation on $E$ it can be regarded as a further fundamental operation. 
Hence in the following we will write $\mathbf E=(E;\oplus,',0,1)$ instead of $\mathbf E=(E;\oplus,0,1)$.

Let $\mathbf E=(E;\oplus,',0,1)$ be an effect algebra and $a,b\in E$. The following facts are well-known:

\noindent{}(F1): By defining $a\leq b$ if there exists some $c\in E$ such 
that $a\oplus c$ 
exists and $a\oplus c=b$, $(E,\leq,',0,1)$ becomes a bounded poset with an antitone involution. 
We call $\leq$ the {\em induced order} of $\mathbf E$. 
 Recall that the element $c$ is unique, if it exists. Then $c$ is equal to $(a \oplus b')'$ 
and it is denoted by $b \ominus a$.
$\mathbf E$ is called a {\em lattice effect algebra} if $(E,\leq)$ is a lattice.

\noindent{}(F2): $a\oplus b$ exists if and only if $a\leq b'$.

\noindent{}(F3): $a\oplus0$ and $0\oplus a$ exist and $a\oplus0=0\oplus a=a$.

\noindent{}(F4): $(a')'=a$.

We recall Proposition~1.8.6 from \cite{dvurec}:

\begin{Proposition}\label{prop1}
Let $\mathbf E=(E;\oplus,',0,1)$ be a lattice effect algebra, $\vee$ and $\wedge$ denote its lattice operations and $a,b,c\in E$. If $a\oplus c$ and $b\oplus c$ exist then $(a\wedge b)\oplus c=(a\oplus c)\wedge(b\oplus c)$.
\end{Proposition}

The following concepts were introduced in \cite{dvurec}.

\begin{Definition}\label{mvef} 
A {\em lattice orthoalgebra} is a lattice effect algebra $\mathbf E=(E,\oplus,',0,1)$ satisfying condition {\hspace{1pt}\rm(E5)} for all $x\in E$:

\noindent{}{\hspace{0pt}\rm (E5)} If $x\oplus x$ exists then $x=0$.

An {\em {\rm MV}-effect algebra} 
is a lattice effect algebra $\mathbf E$ such that 
$(x\wedge y')\oplus y=(y\wedge x')\oplus x$ for all $x,y\in E$.
\end{Definition}

Recall that Rie\v canov\'a \mbox{(see \cite{Riecanova})} showed that every lattice effect algebra is the set-theoretic union of maximal subalgebras which are {\rm MV}-effect algebras, so-called {\em blocks}, and therefore is itself an {\rm MV}-effect algebra if and only if it consists of one block only.

\begin{Definition}\label{def2}\label{def1}
An {\em  effect groupoid} is an algebra $\mathbf R=(R;\cdot,',0,1)$ 
of type $(2,1,0,0)$ satisfying conditions {\rm(NG0)} -- {\rm(NG8)}:

\begin{enumerate}[\hspace{1pt}(NG1)]
\item[{\rm(NG0)}] $(R,\cdot,1)$ is a groupoid with unit $1$;
\item[{\rm(NG1)}] $x=x''$;
\item[{\rm(NG2)}]  $x\cdot 0=0\cdot x=0$; 
\item[{\rm(NG3)}] $0'=1$;
\item[{\rm(NG4)}] $x\cdot (y\cdot x')=0= (y\cdot x')\cdot x$; 
\item[{\rm(NG5)}] $x\cdot y=  y\cdot [(y'\cdot x')'\cdot x']'$;
\item[{\rm(NG6)}] $x\cdot (y'\cdot x)'=(y'\cdot x)'\cdot x=(x'\cdot y)'\cdot y$;
\item[{\rm(NG7)}]  $[(x\cdot y')'\cdot y']'\cdot z=%
[((x\cdot z)\cdot (y\cdot z)')'\cdot (y\cdot z)']'$;
\item[{\rm(NG8)}]  If $x'\cdot y'=0$ and $(x\cdot y)'\cdot z'=0$ then 
$y'\cdot z'=0$, $x'\cdot (y\cdot z)'=0$ and $(x\cdot y)\cdot z=x\cdot (y\cdot z)$.
\end{enumerate}
A {\em sub-effect groupoid of\/ $\mathbf R$} is a subset $Q\subseteq R$ such that 
$0, 1\in Q$ and $a, b\in Q$ implies $a\cdot b\in Q$ and $a'\in Q$. 
\end{Definition}

The following theorem shows that to every lattice effect algebra there can be assigned an effect  groupoid in some natural way. 

\begin{Theorem}\label{mainleatoneg}
Let $\mathbf E=(E;\oplus,',0,1)$ be a lattice effect algebra with lattice operations $\vee$ and $\wedge$ and put
\[
x\cdot y:=((x'\wedge y)\oplus y')'
\]
for all $x,y\in E$. Then $x\cdot y$ is well-defined because of $x'\wedge y\leq y$ and, moreover, 
$\mathbb R(\mathbf E):=(E;\cdot,',0,1)$ is an effect groupoid.
\end{Theorem}

\begin{proof}
Let $a, b\in E$. Since  $(E,\leq,',0,1)$ is a bounded poset with 
an antitone involution we have $a''=a$,  $a\leq b$ implies $b'\leq a'$, and $0'=1$. Moreover, 
$$
\begin{array}{@{}l@{\,}c@{\,}l}
a\cdot 1 & = & ((a'\wedge1)\oplus1')'=(a'\oplus0)'=(a')'=a\text{ and} \\
1\cdot a & = & ((1'\wedge a)\oplus a')'=((0\wedge a)\oplus a')'\\
&=&(0\oplus a')'=(a')'=a.
\end{array}
$$
Hence $(E,\cdot,1)$ is a groupoid with neutral element. Moreover, 
\begin{eqnarray*}
a\cdot 0 & = & ((a'\wedge0)\oplus0')'=(0\oplus1)'=1'=0\text{ and} \\
0\cdot a & = & ((0'\wedge a)\oplus a')'=(a\oplus a')'=1'=0.
\end{eqnarray*}
Also, $a\leq b$ implies $b\ominus a=(b'\oplus a)'=((b'\wedge a')\oplus a)'=b\cdot a'$. 
By (E3) we have that $a'\wedge b'=b'$ implies $b\oplus(a'\wedge b')=b\oplus b'=1$ 
and, conversely, if $b\oplus(a'\wedge b')=1$ then $a'\wedge b'=b'$. Hence the following (denoted by ($\dagger$)) are equivalent: 
\begin{eqnarray*}
a\leq b, b'\leq a', a'\wedge b'=b', b\oplus(a'\wedge b')=1,\\
(a'\wedge b')\oplus b=1, ((a'\wedge b')\oplus b)'=0, a\cdot b'=0.
\end{eqnarray*}

Let $a, b\in E$.  Then $a\wedge b=(a'\cdot b)'\cdot b$. Namely, from the definition of $\cdot$ we obtain that 
$(a'\cdot b)'=(a\wedge b)\oplus b'$. It follows that $a\wedge b=(a'\cdot b)'\ominus b'=(a'\cdot b)'\cdot b$.

Altogether, the conditions  (NG0) -- (NG3) of Definition~\ref{def1} are valid.
Now it remains to prove  the conditions   (NG4) -- (NG8).

\noindent{}(NG4): We have $a\leq a\oplus(b'\wedge a')=(b'\wedge a')\oplus a=(b\cdot a')'$ and hence $a\cdot (b\cdot a')=0$ 
according to  ($\dagger$). Moreover, 
$(b\cdot a') \cdot a= ( (b\cdot a')'\wedge a\oplus a')'=(a\oplus a')'=1'=0$. 

\noindent{}(NG5): We know that $b'\leq a\vee b'= (a'\wedge b)'$. In this case 
$a\cdot b=((a'\wedge b)\oplus b')'=(b'\oplus (a'\wedge b)'')'=%
(b'\wedge (a'\wedge b)'\oplus (a'\wedge b)'')'=b\cdot [(a'\wedge b)']=b\cdot [(b\wedge a')']%
=b\cdot [(b'\cdot a')'\cdot a']'$. 

\noindent{}(NG6): Since  $a\wedge b=(a'\cdot b)'\cdot b$ we have from the commutativity 
of $\wedge$ that  $(b'\cdot a)'\cdot a=(a'\cdot b)'\cdot b$. From the fact that 
$b'\cdot a=((b\wedge a)\oplus a')'\leq a$ we know that 
$a$ and $(b'\cdot a)'$ are both in some block 
of the lattice effect algebra $\mathbf E$ (see \cite{Riecanova}) and hence 
$[(b'\cdot a)'\cdot a]'=((b'\cdot a)\wedge a)\oplus  a'=%
 a'\wedge (b'\cdot a)'\oplus  (b'\cdot a)=%
[a \cdot (b'\cdot a)']'$. 

\noindent{}(NG7): Using Proposition~\ref{prop1} we have   
$[(a\cdot b')'\cdot b']'\cdot c= (((a'\wedge b')\wedge c)\oplus c')'%
=((a'\wedge b'\wedge c)\oplus c')'%
=(((a'\wedge c)\wedge(b'\wedge c))\oplus c')'=%
[((a'\wedge c)\oplus c')\wedge((b'\wedge c)\oplus c')]'%
=[((a\cdot c)\cdot (b\cdot c)')'\cdot (b\cdot c)']'$.

\noindent{}(NG8): Assume $a'\cdot b'=0$ and $(a\cdot b)'\cdot c'=0$. Then there exists $a'\oplus b'$, $a\cdot b=(a'\oplus b')'$ and there exists $(a'\oplus b')\oplus c'$. Hence there exist $b'\oplus c'$ and $a'\oplus(b'\oplus c')$ and $(a'\oplus b')\oplus c'=a'\oplus(b'\oplus c')$. This shows $b'\leq c$, $b'\cdot c'=0$, 
$(b'\oplus c')'=b\cdot c$ and $a'\cdot (b\cdot c)'=0$. Therefore
\smallskip
$
\begin{array}{@{}r@{\, }c@{\, }l}
(a\cdot b)\cdot c&=&((a'\oplus b')\oplus c')'\\
&=&(a'\oplus(b'\oplus c'))'=a\cdot (b\cdot c).
\end{array}
$\end{proof}

Now we show that to every effect  groupoid we can assign a lattice effect algebra in some natural way.

\begin{Theorem}\label{th2}
Let $\mathbf R=(R;\cdot,',0,1)$ be an effect  groupoid
and for $x,y\in R$ put
\[
x\oplus y:=(x'\cdot y')',\text{ provided }x\cdot y=0. 
\]
Then $\mathbb E(\mathbf R):=(R;\oplus,',0,1)$ is a lattice effect algebra.
\end{Theorem}
\begin{proof}
Let $a,b,c\in R$.

\noindent{}(E1): Assume $a\oplus b$ exists. Then $a\cdot b=0$ according to the definition of $\oplus$ and hence $b'\cdot a' =a'\cdot [(a\cdot b)'\cdot b]'=%
a'\cdot [0'\cdot b]'=a'\cdot [1\cdot b]'=a'\cdot b'$ according to (NG5), (NG3) and 
(NG0). Moreover, $b=1\cdot b= (a\cdot b)'\cdot b=(b'\cdot a')'\cdot a'$ by (NG6). 
Hence  $b\cdot a= ((b'\cdot a')'\cdot a')\cdot a=0$ by (NG4). It follows that 
$b\oplus a$ exists and $a\oplus b=b\oplus a$.

\noindent{}(E2): Assume $a\oplus b$ and $(a\oplus b)\oplus c$ exist. Then $a\cdot b=0$ and $(a'\cdot b')'\cdot c=0$ according to  the definition of $\oplus$. Hence 
$b\cdot c=0$, $a\cdot (b'\cdot c')'=0$ and 
$(a'\cdot b')\cdot c'=a'\cdot (b'\cdot c')$ according to (NG8). 
This finally implies that there exist $b\oplus c$ and $a\oplus(b\oplus c)$ and
\[
(a\oplus b)\oplus c=((a'b')c')'=(a'(b'c'))'=a\oplus(b\oplus c).
\]
\noindent{}(E3): If $a\oplus b$ exists and $a\oplus b=1$ then $a\cdot b=0$, i.e., 
$b\cdot a=0$ according to (E1), 
and $(a'\cdot b')'=1$ according to the definition of $\oplus$ and 
hence $a'\cdot b'=0$. It follows 
$a=a\cdot 1= a\cdot (b\cdot a)'= (a'\cdot b')'\cdot b'=1\cdot b'=b'$ 
 according to (NG0) and (NG6).  On the other hand, $a\cdot a'=0$ 
 and $a'\cdot a=0$ by (NG4). Hence $a\oplus a'$ exists and  
 therefore $a\oplus a'=(a'\cdot a)'=0'=1$ according to (NG3). 
 On that matter $a\oplus b=1$ if and only if $b=a'$.

\noindent{}(E4): If $a\oplus 1$ exists then, using (NG0), $a=a\cdot 1=0$ and hence $a=0$.

Hence $\mathbb E(\mathbf R)$ is an effect algebra. Let $\leq$ denote 
its induced order. Then $a\leq b'$ if and only if $a\oplus b$ exists if and only 
if $a\cdot b=0$. It is enough to check that the operation $\wedge$ 
defined by $a\wedge b=(a'\cdot b)'\cdot b$ is a meet with respect to $\leq$. 
From (NG4) we obtain that $((a'\cdot b)'\cdot b)\cdot b'=0$. Hence 
$a\wedge b\leq b$. Since also  $a\wedge b=(b'\cdot a)'\cdot a$ by (NG6) 
we obtain $a\wedge b\leq a$. Let $x\in R$, $x\leq a$ and $x\leq b$. 
Then $a'\cdot x=0=b'\cdot x$. It follows by (NG7) that $(a\wedge b)'\cdot x=%
[(a'\cdot b'')'\cdot b'']'\cdot c=%
[((a'\cdot c)\cdot (b'\cdot c)')'\cdot (b'\cdot c)']'=[(0\cdot 0')'\cdot 0']'= 
[0'\cdot 0']'=1'=0$. Therefore $x\leq a\wedge b$ .
\end{proof}

Next we show that the described correspondence between lattice effect algebras and effect groupoids is one-to-one.

\begin{Theorem}
Let $\mathbf E=(E;\oplus,',0,1)$ be a lattice effect algebra. Then $\mathbb E(\mathbb R(\mathbf E))=\mathbf E$.
\end{Theorem}

\begin{proof}
Let $\mathbb R(\mathbf E)=(E;\cdot,',0,1)$, 
$\mathbb E(\mathbb R(\mathbf E))=(E;\oplus_1,',0,1)$ and $a,b\in E$. 
Then the following are equivalent: $a\oplus_1b$ exists, $a\cdot b=0$, $a\oplus b$ exists. 
If this is the case then $a\oplus_1 b=(a'\cdot b')'=(a\wedge b')\oplus b=a\oplus b$.
\end{proof}

\begin{Theorem}\label{th1}
Let $\mathbf R=(R; \cdot,',0,1)$ be an effect groupoid. Then $\mathbb R(\mathbb E(\mathbf R))=\mathbf R$.
\end{Theorem}

\begin{proof}
Let $\mathbb E(\mathbf R)=(R;\oplus,',0,1)$, 
$\mathbb R(\mathbb E(\mathbf R))=(R;\cdot_1,',0,1)$,  and $a,b\in R$. From 
(NG7),  (NG4),  (NG3) and (NG0) we have 
\smallskip

$\begin{array}{r c l}
a\cdot_1 b&=&((a'\wedge b)\oplus b')'=(a'\wedge b)'\cdot b\\
&=&[(a\cdot b'')'\cdot b'']'\cdot b\\
&=&%
[((a\cdot b)\cdot (b'\cdot b)')'\cdot (b'\cdot b)']'\\
&=&%
[((a\cdot b)\cdot 0')'\cdot 0']'=a\cdot b.
\end{array}$

\end{proof}

Now we can characterize lattice orthoalgebras by means of effect groupoids 
as follows:

\begin{Theorem}
A lattice effect algebra $\mathbf E=(E,\oplus,',0,1)$ is a lattice orthoalgebra if and only if its corresponding effect groupoid $\mathbf R=(E,\cdot,',0,1)$ is  idempotent, i.e., 
\ it satisfies the identity $x\cdot x= x$.
\end{Theorem}

\begin{proof}
If $\mathbf E$ is a lattice orthoalgebra then $a\wedge a'=0$ for 
all $a\in E$ (cf.\ Example~4.3 in \cite{CHK}). Hence
\[
a\cdot a=((a'\wedge a)\oplus a')'=(0\oplus a')'=(a')'=a
\]
according to (F3) and (F4). Conversely, assume that $\mathbf R$ satisfies 
$x\cdot x=x$. Let $a\in E$ such that $a\oplus a$ exists. 
Then $a\leq a'$ and hence $a=a\cdot a=((a'\wedge a)\oplus a')'$ according 
to Theorem \ref{th2}. It follows that  $0\oplus a'=(a'\wedge a)\oplus a'$, 
i.e., using (E3) we obtain $0=a'\wedge a=a$. 
\end{proof}

Similarly, we can characterize {\rm MV}-effect algebras by 
means of effect groupoids as follows:

\begin{Corollary}\label{cor1}
A lattice effect algebra $\mathbf E=(E,\oplus,',0,1)$ is an {\rm MV}-effect algebra if and only if its corresponding effect groupoid $\mathbf R=(E,\cdot,',0,1)$ is commutative.
\end{Corollary}
\begin{proof}
If $\mathbf E$ is an {\rm MV}-effect algebra then
\[
x\cdot y=((x'\wedge y)\oplus y')'=((y'\wedge x)\oplus x')'= y\cdot x
\]
and if, conversely, $\cdot $ is commutative then
\[
(x\wedge y')\oplus y=(x'\cdot y')'=(y'\cdot x')'=(y\wedge x')\oplus x
\]
and hence $\mathbf E$ is an {\rm MV}-effect algebra.
\end{proof}

It is an easy observation that a commutative effect groupoid is 
associative  (cf.\ \cite[Theorem 2]{Ko}). Due to 
Rie\v canov\'a's theorem (cf.\ \cite[Theorem 3.2]{Riecanova}) we conclude 

\begin{Corollary}\label{cor1x}
Every effect groupoid is a set-theoretic union 
of associative and commutative sub-effect groupoids.
\end{Corollary}

\section{The groupoid-based logic for lattice effect algebras}\label{LDM}

We know that the logic associated to MV-algebras is already 
desribed as many-valued Lukasiewicz logic and its axioms and 
reference rules are well-known, the same can be said on the logic 
induced by orthomodular lattices (see e.g. \cite{Abbott2}). The 
previous Corollaries \ref{cor1} and \ref{cor1x} motivate us to set up an appropriate 
logic also for lattice effect algebras. Of course, we will formulate the 
axioms and rules in the language of effect groupoids as derived in the previous part.

In what follows, similarly as in \cite{suzuki}, 
we denote propositional variables by  $p; q; r; \dots$, the logical binary connective by $\cdot$, 
the logical unary connective negation by $\neg$, and two logical constants $\bot$ and $\top$ where $\bot$  
stands for the contradiction and $\top$ stands for the tautology. 
So formulae are inductively defined by the following BNF:
$$\phi::=p\mid \phi\cdot \phi\mid \neg \phi \mid \bot \mid \top.$$
We denote formulae by $\phi, \psi, \chi, \mu, \dots$ and 
let $\Phi$ and $\Lambda$ be the set of all propositional variables 
and the set of all formulae. Let 
 $\Gamma, \Delta, \Sigma, \Pi$ be arbitrary (possibly empty) 
 finite lists of formulae, $\varphi, \nu$ a list of at
most one formula. 
A logical consequence relation $\lcr$ is a binary relation on $\Lambda$. 
We may  interpret  $\phi\lcr \psi$  as \uv{if $\phi$ then $\psi$.}
 So we call the left-hand formulae premises and the right-hand formulae conclusions. We may
sometimes call logical consequences sequents. In the following 
\begin{tabular}{r l}
\begin{tabular}{c}
{$\Gamma\lcr \varphi$}\quad {$\Delta\lcr \nu$}\\
\hline 
 \mbox{ \rule{0cm}{0.41cm}$\phi\eqcr \psi$}
\end{tabular}&
\end{tabular}
will be short for the  two rules: 
\begin{tabular}{r l}
\begin{tabular}{c}
{$\Gamma\lcr \varphi$}\quad {$\Delta\lcr \nu$}\\
\hline 
 \mbox{ \rule{0cm}{0.41cm}$\phi\lcr \psi$}
\end{tabular}&
\end{tabular} and 
\begin{tabular}{r@{.} l}
\begin{tabular}{c}
{$\Gamma\lcr \varphi$}\quad {$\Delta\lcr \nu$}\\
\hline 
 \mbox{ \rule{0cm}{0.41cm}$\psi\lcr \phi$}
\end{tabular}&
\end{tabular}
We now  introduce a sequent calculus 
($L_{LEA}$) given as follows.\\[0.2cm]
{\parindent0pt
\begin{tabular}{r@{\ }l}
\begin{tabular}{c}
{$\phi\lcr \neg\psi$}\\
\hline 
$\psi\lcr \neg\phi$
\end{tabular}&($\neg$-r)
\end{tabular}\phantom{x}
\begin{tabular}{r@{\ }l}
\begin{tabular}{c}
{$\phi\lcr \psi \quad \psi\lcr \phi $}\\
\hline 
\mbox{ \rule{0cm}{0.41cm}$\phi\cdot \mu\lcr \psi\cdot \mu$}
\end{tabular}&(itm1)
\end{tabular}
\\[0.1cm]
\begin{tabular}{r@{\ }l}
\begin{tabular}{c}
\phantom{$\phi\lcr \psi$}\\
\hline 
\mbox{\rule{0cm}{0.51cm}$\neg\neg\phi\eqcr \phi$}
\end{tabular}&(DN)
\end{tabular}\phantom{}\hskip-0.56543251em
\begin{tabular}{r@{\ }l}
\begin{tabular}{c}
{$\phi\lcr \psi \quad \psi\lcr \phi $}\\
\hline 
\mbox{ \rule{0cm}{0.41cm}$\mu\cdot \phi\lcr \mu\cdot \psi$}
\end{tabular}&(itm2)
\end{tabular}
\\[0.2cm]
\begin{tabular}{r@{\ }l}
\begin{tabular}{c}
{$\phi\cdot \neg\psi\lcr \bot$}\\
\hline 
\mbox{ \rule{0cm}{0.41cm}$\phi\lcr \psi$}
\end{tabular}&(m-$\bot$)
\end{tabular}
\begin{tabular}{r@{\ }l}
\begin{tabular}{c}
{$\phi\lcr \psi$}\\
\hline 
\mbox{ \rule{0cm}{0.41cm}$\phi\cdot \neg\psi\lcr \bot$}
\end{tabular}&($\bot$-m)
\end{tabular}
\\[0.2cm]
\begin{tabular}{r@{\ }l}
\begin{tabular}{c}
\phantom{$\phi\lcr \psi$}\\
\hline 
 \mbox{\rule{0cm}{0.51cm}$\top\cdot\phi\eqcr \phi$}
\end{tabular}&(1-l)
\end{tabular}
\hskip-0.6543251em
\phantom{\,}
\begin{tabular}{r@{\ }l}
\begin{tabular}{c}
\phantom{$\phi\lcr \psi$}\\
\hline 
 \mbox{\rule{0cm}{0.51cm}$\phi\cdot\top\eqcr \phi$}
\end{tabular}&(1-r)
\end{tabular}\\[0.2cm]
\begin{tabular}{r@{\ }l}
\begin{tabular}{c}
\phantom{$\phi\lcr \psi$}\\
\hline 
 \mbox{\rule{0cm}{0.51cm}$\bot\cdot\phi\eqcr \bot$}
\end{tabular}&(0-l)
\end{tabular}
\hskip-1.234587643251em\phantom{\,}
\begin{tabular}{r@{\ }l}
\begin{tabular}{c}
\phantom{$\phi\lcr \psi$}\\
\hline 
 \mbox{\rule{0cm}{0.51cm}$\phi\cdot\bot\eqcr \bot$}
\end{tabular}&(0-r)
\end{tabular}\\[0.2cm]
\begin{tabular}{r@{\ }l}
\begin{tabular}{c}
\phantom{$\phi\lcr \psi$}\\
\hline \mbox{ \rule{0cm}{0.51cm}$\bot\eqcr \phi \cdot (\psi \cdot \neg \phi)$}
\end{tabular}&(ol)
\end{tabular}\\[0.2cm]
\begin{tabular}{r@{\ }l}
\begin{tabular}{c}
\phantom{$\phi\lcr \psi$}\\
\hline \mbox{ \rule{0cm}{0.51cm}$\bot\eqcr (\psi \cdot \neg \phi) \cdot \phi $}
\end{tabular}&(or)
\end{tabular}\\[0.2cm]
\begin{tabular}{r@{\ }l}
\begin{tabular}{c}
\phantom{$\phi\lcr \psi$}\\
\hline \mbox{ \rule{0cm}{0.51cm}$\phi\cdot \neg(\neg \psi\cdot \phi)\eqcr %
\neg(\neg \psi\cdot \phi)\cdot \phi$}
\end{tabular}&(cm)
\end{tabular}\\[0.2cm]
\begin{tabular}{r@{\ }l}
\begin{tabular}{c}
\phantom{$\phi\lcr \psi$}\\
\hline \mbox{ \rule{0cm}{0.51cm}$\neg(\neg \psi\cdot \phi)\cdot \phi\eqcr %
\neg(\neg \phi\cdot \psi)\cdot \psi$}
\end{tabular}&(mc)
\end{tabular}\\[0.2cm]
\begin{tabular}{r@{\ }l}
\begin{tabular}{c}
\phantom{$\phi\lcr \psi$}\\
\hline \mbox{ \rule{0cm}{0.51cm}$\phi\cdot \psi\eqcr \psi \cdot \neg(\neg(\neg\psi\cdot \neg\phi) \cdot \neg\phi) $}
\end{tabular}&(ocm)
\end{tabular}\\[0.13cm]

Denote $\chi(\phi,\psi):=\neg(\neg(\phi\cdot \neg\phi)\cdot \neg\phi)$. \\[0.13cm]
\begin{tabular}{r@{\ }l}
\begin{tabular}{c}
\phantom{$\phi\lcr \psi$}\\
\hline \mbox{ \rule{0cm}{0.51cm}$\chi(\phi,\psi)\cdot \mu\eqcr %
\chi(\phi\cdot \mu,\psi\cdot \mu)$}
\end{tabular}&(mds)
\end{tabular}\\[0.2cm]

\begin{tabular}{r@{\ }l}
\begin{tabular}{c}
{$\phi\lcr \psi$}\phantom{x}{$\psi\lcr \chi$}\\
\hline 
\mbox{\rule{0cm}{0.51cm}$\phi\lcr \chi$}
\end{tabular}&(cut)
\end{tabular}\\[0.2cm]

\begin{tabular}{r@{\ }l}
\begin{tabular}{c}
{$\neg\phi\lcr \psi$}\phantom{xx}{$\neg(\phi\cdot \psi)\lcr \mu$}\\
\hline 
\mbox{ \rule{0cm}{0.51cm}$\neg\psi\lcr \mu$}
\end{tabular}&(ass1)
\end{tabular}
\\[0.5cm]
\begin{tabular}{r@{\ }l}
\begin{tabular}{c}
{$\neg\phi\lcr \psi$}\phantom{xx}{$\neg(\phi\cdot \psi)\lcr \mu$}\\
\hline 
\mbox{ \rule{0cm}{0.51cm}$\neg\phi\lcr \psi\cdot \mu$}
\end{tabular}&(ass2)
\end{tabular}
\\[0.2cm]
\begin{tabular}{r@{\ }l}
\begin{tabular}{c}
{$\neg\phi\lcr \psi$}\phantom{xx}{$\neg(\phi\cdot \psi)\lcr \mu$}\\
\hline 
\mbox{ \rule{0cm}{0.51cm}$(\phi\cdot \psi)\cdot \mu \eqcr \phi\cdot (\psi\cdot \mu)$}
\end{tabular}&(ass3)
\end{tabular}
\\[0.2cm]

}
 The groupoid-based logic for lattice effect algebras 
 is the collection of all sequents derivable in $L_{LEA}$, 
denoted by ${\mathbf L}_{LEA}$.

Using previous axioms and inductive steps, 
we can derive several useful rules as follows.

\begin{proposition}\label{seqcDMP} 
In the sequent calculus $L_{LEA}$, we can derive the following theorems and inference rules.
\end{proposition}
\noindent%
\begin{tabular}{r@{\ }l}
\begin{tabular}{c}
\phantom{$\phi\lcr \psi$}\\
\hline 
\mbox{\rule{0cm}{0.51cm}$\phi\lcr \phi$}
\end{tabular}&(Ax)
\end{tabular}
\phantom{xxxxx}\\[0.25cm]
\begin{tabular}{r@{\ }l}
\begin{tabular}{c}
\phantom{$\phi\lcr \psi$}\\
\hline 
\mbox{\rule{0cm}{0.51cm}$\top\eqcr \neg\bot$}
\end{tabular}&(tnb)
\end{tabular}
\phantom{x}
\begin{tabular}{r@{\ }l}
\begin{tabular}{c}
\phantom{$\phi\lcr \psi$}\\
\hline 
\mbox{\rule{0cm}{0.51cm}$\neg\top\eqcr \bot$}
\end{tabular}&(ntb)
\end{tabular}\\[0.25cm]
\begin{tabular}{r@{\ }l}
\begin{tabular}{c}
\phantom{$\phi\lcr \psi$}\\
\hline 
\mbox{\rule{0cm}{0.51cm}$\bot\lcr \phi$}
\end{tabular}&(bot)
\end{tabular}
\phantom{xxxxsx}\begin{tabular}{r@{\ }l}
\begin{tabular}{c}
\phantom{$\phi\lcr \psi$}\\
\hline 
\mbox{\rule{0cm}{0.51cm}$\phi\lcr \top$}
\end{tabular}&(top)
\end{tabular}\\[0.25cm]
\begin{tabular}{r@{\ }l}
\begin{tabular}{c}
{$\phi\lcr \psi$}\\
\hline 
\mbox{\rule{0cm}{0.51cm}$\neg\psi\lcr \neg\phi$}
\end{tabular}&($\neg$-$\neg$)
\end{tabular}
\phantom{xx}
\begin{tabular}{r@{\ }l}
\begin{tabular}{c}
{$\neg\phi\lcr \psi$}\\
\hline 
\mbox{\rule{0cm}{0.51cm}$\neg\psi\lcr \phi$}
\end{tabular}&($\neg$-l)
\end{tabular}
\begin{proof} (Ax) follows from (ol) when $\psi=\top$ and then we apply (m-$\bot$) when $\psi=\phi$. 
First half of (tnb) follows from (0-r) and then we apply  (m-$\bot$). Second half of (tnb) follows by 
applying first (1-l) when $\phi=\neg\bot\cdot\neg \top$ and then applying (ol) when $\phi=\top$ and $\psi=\neg\bot$ and using (cut) 
and (m-$\bot$). 
By the same considerations we can get (ntb). 

To show (top) we first apply (itm2) using (ntb) and then (0-r). This yields that $\phi\cdot \neg \top\lcr \bot$ which in turn 
from (m-$\bot$) and (cut) gives (top). (bot) immediately follows from (0-l) and then we again apply  (m-$\bot$).

($\neg$-$\neg$)  and ($\neg$-l) follow by the same 
considerations as in \cite[Proposition 2.1]{suzuki}.
\end{proof}

The key question is the soundness of the given logic, i.e. its correspondence to the given algebraic structure. 
Fortunately, we are able to prove the following theorem.

\begin{theorem} \label{soundDMP} (Soundness). The groupoid-based 
logic for lattice effect algebras ${\mathbf L}_{LEA}$ 
is sound for the class of  effect groupoids. That is, for every sequent 
$\phi\lcr \psi$ in ${\mathbf L}_{LEA}$, $s_{\phi}\leq t_{\psi}$ is valid 
on all  effect groupoids ${\mathbf R}$, where $s_{\phi}$ and
$t_{\psi}$ are the corresponding term functions for $\phi$ and  $\psi$ and $\leq$ is 
the order in the corresponding lattice effect algebra $\mathbb E(\mathbf R)$.
\end{theorem}
\begin{proof} Let  ${\mathbf R}=(R;\cdot,',0,1)$ be an effect groupoid. 
The axiom (DN-l) follows from the fact that $a=a''$. For the inductive steps,
($\neg$-r) follows from the fact that $a\leq b \Rightarrow b'\leq a'$\ 
and $a=a''$, (itm1) and (itm2) follow from the 
antisymmetry of $\leq$, and (cut) follows from the transitivity of $\leq$. 
Since  $0$ is the bottom element   of $\mathbb E(\mathbf R)$  we have as in the proof 
of Theorem \ref{th2} that $a\leq b$ if and only if $a\cdot b'\leq 0$. This 
yields that both  (m-$\bot$) and ($\bot$-m) are valid. Since 
$1\cdot a\leq a$ and $a\leq 1\cdot a$, and  
$a\cdot 1\leq a$ and $a\leq a\cdot 1$ we get (1-l) and (1-r). Similarly, 
since $0\cdot a\leq 0$ and $0\leq 0\cdot a$, and  
$a\cdot 0\leq 0$ and $0\leq a\cdot 0$ we get (0-l) and (0-r). 
The axioms (ol) and (or) follow immediately from (NG4). Further, 
the axioms (cm) and (mc) are valid by (NG6) and the axiom (ocm) 
is valid by (NG5). The axiom (mds) follows from (NG7) and the axioms 
(ass1), (ass2) and (ass3) follow from (NG8). 

\end{proof} 


The second important property of a given logic is its completeness. 
Similarly as for the classical logic, we prove the following assertion by using of the corresponding Lindenbaum-Tarski algebra.

\begin{theorem} \label{compDMP} 
(Completeness). The groupoid-based logic for lattice effect algebras 
${\mathbf L}_{LEA}$ 
is complete with respect to the class of effect groupoids.
\end{theorem}
\begin{proof}  As in \cite[Theorem 3.7]{suzuki} or in \cite[Theorem II.3]{TODDMA3}, 
we take the Lindenbaum-Tarski algebra ${\mathbf R}_{LEA}$ 
for  ${\mathbf L}_{LEA}$. That is, we take the quotient of  $\Lambda$ with respect
to the equivalence relation $\equiv$, defined by $\phi\equiv\psi$  $\ekviv$ 
$\phi\lcr \psi$ in ${\mathbf L}_{LEA}$ and 
$\psi\lcr \phi$ in ${\mathbf L}_{LEA}$. It is plain that 
$\equiv$ is really an equivalence relation. On this quotient
set $\Lambda/_{\equiv}$, we can define
\begin{itemize}
\item $0:=[\bot]_{\equiv}$, $[\phi]_{\equiv}':= [\neg\phi]_{\equiv}$, 
\item $1:=[\top]_{\equiv}$, $[\phi]_{\equiv}\cdot [\psi]_{\equiv}=%
[\phi\cdot \psi]_{\equiv}$. 
\end{itemize}

First, we have to verify that the definitions of $\cdot$ and $'$ do not 
depend on  representatives. Assume that $\phi\lcr \overline{\phi}$, 
$\overline{\phi}\lcr {\phi}$, $\psi\lcr \overline{\psi}$ and  
$\overline{\psi}\lcr {\psi}$. Using Proposition \ref{seqcDMP}, we get 
by ($\neg-\neg$) that $\neg\overline{\phi}\lcr \neg{\phi}$ and 
$\neg\phi\lcr \neg\overline{\phi}$. Hence 
$\neg\overline{\phi}\equiv \neg{\phi}$. Similarly, we have 
from (itm2) that $\phi\cdot \psi\lcr \phi\cdot \overline{\psi}$ and by (itm1) 
that $\phi\cdot \overline{\psi}\lcr \overline{\phi}\cdot \overline{\psi}$. Using 
(cut) we obtain that $\phi\cdot \psi\lcr \overline{\phi}\cdot \overline{\psi}$. 
By symmetric considerations we obtain that 
$\overline{\phi}\cdot \overline{\psi}\lcr \phi\cdot \psi$, i.e., 
$\phi\cdot \psi\equiv \overline{\phi}\cdot \overline{\psi}$.

It is a transparent task to show that the Lindenbaum-Tarski algebra  
${\mathbf R}_{LEA}=(\Lambda/_{\equiv};\cdot,',0,1)$ is 
 an effect groupoid.  

\end{proof} 

\section*{Acknowledgements}  
This is a pre-print of an article published as \newline 
I. Chajda, H.  L\"anger, J. Paseka, The groupoid-based logic for lattice effect algebras, in: Proceedings of the 47th IEEE International Symposium on Multiple-Valued Logic, Springer, (2017), 230--235, 
doi: 10.1109/ISMVL.2017.15.
The final authenticated version of the article is available online at: 
\newline 
https://ieeexplore.ieee.org/stamp/stamp.jsp?tp=\&\-arnumber=7964996.

All authors acknowledge the support by a bilateral project 
New Perspectives on Residuated Posets  financed by  
Austrian Science Fund (FWF): project I 1923-N25, 
and the Czech Science Foundation (GA\v CR): project 15-34697L. 
J.~Paseka acknowledges the financial support of the Czech Science Foundation
(GA\v CR) un\-der the grant  Algebraic, Many-valued and Quantum Structures for Uncertainty Modelling:  
project  15-15286S.

\end{document}